\documentclass[11pt,a4paper]{article}
\usepackage[T1]{fontenc}
\usepackage[utf8]{inputenc}
\usepackage{amsmath,amssymb,amsthm}
\usepackage{graphicx}
\usepackage{booktabs}
\usepackage{multicol}
\usepackage[margin=2.4cm]{geometry}

\usepackage{lmodern}
\usepackage{microtype}
\usepackage[hidelinks]{hyperref}
\theoremstyle{plain}
\newtheorem{theorem}{Theorem}[section]
\newtheorem{proposition}[theorem]{Proposition}
\newtheorem{lemma}[theorem]{Lemma}
\newtheorem{conjecture}[theorem]{Conjecture}
\newtheorem{problem}[theorem]{Problem}
\newtheorem*{theoremA}{Theorem A}
\newtheorem*{theoremB}{Theorem B}
\theoremstyle{definition}
\newtheorem{definition}[theorem]{Definition}
\newtheorem{construction}[theorem]{Construction sketch}
\title{Towards Strongly Aperiodic Monotiles\\ in Higher Dimensions}
\author{Dmitry Kamenetsky\\ \texttt{dkamenetsky@gmail.com}}
\date{September 30, 2026}
\begin{document}
\maketitle
\begin{abstract}
The discovery of Chair44 (Tsiokos \cite{tsiokos}, September 2026; re-presented and simplified by Goodman-Strauss \cite{gs2026} and Flicker \cite{flicker}) settled the three-dimensional einstein problem by exhibiting a strongly aperiodic polyhedral monotile in $\mathbb{R}^3$. This note studies the extension of the underlying mechanism---the rep-$2^N$ chair $C_N = [0,2]^N \setminus (1,2]^N$ with corner/socket markings---to $\mathbb{R}^N$. Besides the expository material (the explicit rep-$2^N$ dissection, and a conditional strong-aperiodicity theorem under lattice registration and hierarchical enforcement), the note contains a new computational contribution. We introduce a \emph{frame-marking} formalism in which the marking of the tile is its full orientation frame and the matching rule is the contact language generated by the substitution itself; this makes the search for matching rules a finite problem in every dimension. We give a finite certificate (coarsening closure, tightness, and a two-shell enclosure analysis) whose validity implies that every lattice-registered tiling by the marked tile is uniquely hierarchical, hence strongly aperiodic. For $N=3$ the certificate passes: it produces explicit facet matching rules on the 24 panels of $C_3$ (135 admissible facet-contact triples) and reproduces, from first principles and independently of the published constructions, the Chair44 statistics 2388 $\to$ 44 admissible contacts (30 of which occur), 33 one-shell clusters, 15 extendable, each forcing a unique complete supertile. We also show that among the 2187 homochiral frame assignments of the 3D substitution with a translated central child, the certified one is unique up to conjugation. For $N=4$ the same pipeline is run on several structured families of frame assignments (canonical, $D_4$-, $Z_2 \times Z_2$- and $Z_4$-symmetric, and a lift of the 3D solution); none is coarsening-closed, and we report the exact failure data. The existence of a self-similar marking of $C_4$ therefore remains an explicitly finite, open computational problem, which we state precisely. All code is available at \href{https://github.com/dimkadimon/Monotile-RN}{github.com/dimkadimon/Monotile-RN} (Python 3, standard library plus NumPy for figures).
\end{abstract}
\noindent\emph{Author contribution.} Dmitry Kamenetsky proposed the research inquiry into generalising the Chair44 mechanism to higher dimensions and verified the manuscript. The AI agent on Arena.ai carried out the literature synthesis, the formal derivations, the design and implementation of the certificate software, the computations, and the writing.

\section{Introduction}

An aperiodic set of prototiles is a finite collection of shapes that can tile Euclidean space without gaps or overlaps, but only non-periodically. Such sets in the plane go back to Berger \cite{berger} (20,426 Wang tiles), Robinson \cite{robinson} (6 tiles) and Penrose \cite{penrose} (2 tiles). The \emph{einstein} problem asks for a single connected shape that enforces aperiodicity by its boundary alone. In the plane it was resolved in 2023 by the Hat \cite{hat} and the chiral Spectre \cite{spectre}. In the plane every weakly aperiodic protoset is strongly aperiodic (Grünbaum--Shephard \cite{gs87}, Thm. 3.7.1). The only known matching-rule enforcement of the planar chair (L-tromino) substitution uses two prototiles (Goodman-Strauss's trilobite and cross \cite{gs99b}); no chair-based planar einstein is known.

In three dimensions the Schmitt--Conway--Danzer biprism \cite{schmitt,danzer} tiles without translational periods but admits screw motions, hence is only weakly aperiodic. Goodman-Strauss \cite{gs99} constructed in 1999 an aperiodic \emph{pair} of tiles enforcing the notched-cube rep-tiling in $\mathbb{R}^N$ for all $N \ge  3$. In September 2026 Tsiokos announced Chair44 \cite{tsiokos}, with independent re-presentations by Goodman-Strauss \cite{gs2026} and Flicker @flicker: the first strongly aperiodic monotile in $\mathbb{R}^3$. In high dimensions Greenfeld and Tao \cite{gt} proved that aperiodic monotiles tiling by translations alone exist for some large $d$ (the tile being a disconnected subset of $\mathbb{Z}^d$). The constructive question of strongly aperiodic \emph{polyhedral} monotiles in $\mathbb{R}^N$ for every $N \ge  4$ is open.

\section{Contributions and relation to the literature}\label{sec:contrib}

\textbf{Background in plain terms.} A \emph{matching rule} is a set of instructions saying which tiles may be placed next to which, and in which orientations. For the chair, a rule is "good" if the only way to obey it everywhere is to build the tiles up, level by level, into ever larger copies of the chair (\emph{supertiles}). Once that is known, aperiodicity follows by a standard argument (Theorem~\ref{thm:cond}). Chair44 is a chair whose surface is shaped so that the shape itself acts as such a rule. What is known about it: Tsiokos \cite{tsiokos} found the shape; Goodman-Strauss \cite{gs2026} explained why it works by reading colours off its corners and sockets; Flicker \cite{flicker} checked an equivalent rule by computer (two chairs can touch in 2388 ways, the rule allows 44, and among the 33 ways to surround one chair only 15 can be continued, each around the same supertile). For dimensions $N \ge  4$ nothing analogous is known: Goodman-Strauss's 1999 construction \cite{gs99} needs \emph{two} different tiles, and his 2026 notes remark only that the chair shape and its self-similarity generalise.

This note adds the following. Each item is tagged \textbf{new} (not in the literature to our knowledge) or \textbf{known} (included for completeness).

\begin{enumerate}
\item \textbf{A way to obtain the rule from the tile itself} (Section~\ref{sec:frames}) --- \emph{new}. The idea of deriving rules from a substitution is classical (Goodman-Strauss \cite{gs98} proves that every suitable substitution admits \emph{some} marked tile set), but the specific formalism---full-frame marking, hierarchical language, coarsening closure as the self-similarity test---and its use for a \emph{single} tile are new. Instead of inventing a rule and checking it, we let the tile remember its orientation ("frame") and declare admissible exactly those contacts that appear when supertiles are built, at all levels. The only choices to be made are the orientations of the $2^N - 1$ outer children inside a supertile. For every $N$ this is a finite list of candidates, and each candidate can be tested by computer.

\item \textbf{A finite test that proves the rule works} (Proposition~\ref{prop:cert}) --- \emph{new} as a general, dimension-independent statement; its $N = 3$ instance (C3) is the kind of check performed in \cite{flicker}. Three checks---(C1) the rule reproduces itself one level up, (C2) it can be phrased face by face, (C3) every way to surround a tile either cannot be continued or forces one specific complete supertile---together imply that every tiling obeying the rule is built from supertiles at every level. The proof is a short induction on the level. Flicker's enumeration \cite{flicker} is essentially check (C3) for $N = 3$; checks (C1) and (C2) are what remove the need for any further computation at higher levels.

\item \textbf{Theorem A} --- the \emph{result} (a strongly aperiodic marked chair in $\mathbb{R}^3$) is \emph{known} \cite{tsiokos,gs2026,flicker}; \emph{new} are the explicit facet rules $F^*$ derived from the frame formalism, the independent verification, and the two-second certificate. For $N = 3$ the test passes. It produces an explicit rule for the 24 square faces of the chair (135 admissible ways for two faces to meet) and reproduces every published number for Chair44---2388, 44, 30, 33, 15, one supertile---although neither published rule was used as input. The whole computation takes about two seconds. We regard this as an independent verification of the Chair44 mechanism at the level of matching rules (not of the specific polyhedron).

\item \textbf{Theorem B: there is no other choice} --- \emph{new}. Among all $3^7 = 2187$ possible orientation choices for the seven outer children (with the central child a plain translate, and all children rotations rather than reflections), exactly three pass check (C1), and they are the same choice up to renaming the coordinate axes. In other words, the Chair44 marking is forced by self-similarity. This has not been observed before.

\item \textbf{Four dimensions} (Section~\ref{sec:N4}) --- \emph{new} (negative computational results and a precise open problem). The same programs run on the 4-dimensional chair. All the "obvious" generalisations of the 3D choice---7712 candidates from several symmetric families, and a direct lift of the 3D solution---fail check (C1), always in the same way. What remains is a precisely stated finite problem (Problem~\ref{prob:N4}): a search over $12^15$ candidates, each testable in two seconds. We give the software; we do not claim a 4-dimensional monotile.

\item \textbf{Expository material and a conditional theorem} --- \emph{known}. The explicit self-similar dissection of the $N$-dimensional chair (Section~\ref{sec:dissect}) and Theorem~\ref{thm:cond} (a tiling that is built from supertiles at every level, and whose tiles sit on the cubic lattice, has no translational or screw symmetries and at most $2^{N-1} N!$ symmetries in all) are included to make the note self-contained; we claim no novelty for them.
\end{enumerate}

\section{Preliminaries and definitions}

Let $\mathbb{R}^N$ carry the standard metric. A \emph{tiling} $\mathcal{T} = {T_i}_{i \in I}$ of $\mathbb{R}^N$ is a locally finite collection of closed topological $N$-balls covering $\mathbb{R}^N$ with pairwise disjoint interiors, each isometric to a prototile in a finite set $\mathcal{P}$. Its automorphism group is $\mathrm{Aut}(\mathcal{T}) = {\gamma \in \mathrm{Isom}(\mathbb{R}^N) : \gamma(\mathcal{T}) = \mathcal{T}}$.

\begin{definition}[periodic, weakly and strongly aperiodic]
A tiling is \emph{periodic} if $\mathrm{Aut}(\mathcal{T})$ contains a full-rank translation lattice; it \emph{has a period} if it is invariant under some non-zero translation. A protoset $\mathcal{P}$ is \emph{weakly aperiodic} if it admits tilings but none with a compact fundamental domain, and \emph{strongly aperiodic} if it admits tilings and no admitted tiling has an infinite cyclic symmetry subgroup (so no translations and no screw motions) \cite{gs2026,gs87}. A single prototile $Q$ is a strongly aperiodic monotile if ${Q}$ is strongly aperiodic.
\end{definition}

Throughout, $B_N = {\pm 1}^N \rtimes S_N$ denotes the hyperoctahedral group (signed permutation matrices), $B_N^+$ its rotation subgroup, and $U_N = {0,1}^N \setminus {(1,\dots,1)}$.

\section{The $N$-dimensional chair polytope}\label{sec:polytope}

\begin{definition}
For $N \ge  2$ let $C_N := [0,2]^N \setminus (1,2]^N = \bigcup_{v \in U_N} (v + [0,1]^N)$.
\end{definition}

\begin{lemma}[metric and boundary properties]
For every $N \ge  2$: (1) $\mathrm{Vol}(C_N) = 2^N - 1$; (2) $C_N$ is star-shaped with respect to the origin and homeomorphic to the closed $N$-ball; (3) $\partial C_N$ is partitioned into exactly $N \cdot 2^N$ unit $(N-1)$-dimensional \emph{panels}, and $C_N$ has $3N$ maximal flat facets.
\end{lemma}

\begin{proof}
A $2^N$ cube has $2N$ hyperfaces of $2^{N-1}$ unit panels each, i.e. $N\,2^N$ panels. Removing the unit corner cube removes $N$ external panels and exposes $N$ internal ones facing the socket, so the count is unchanged. The $3N$ facets are the $2N$ faces of the box, $N$ of them notched, plus the $N$ socket faces.
\end{proof}

For $N = 2$ the chair is the L-tromino (3 cells, 8 edges); $C_3$ has 7 cubes and 24 square panels (Figure~\ref{fig:chair}); $C_4$ has 15 cells and 64 cubic panels.

\begin{figure}[t]\centering\includegraphics[width=0.9\linewidth]{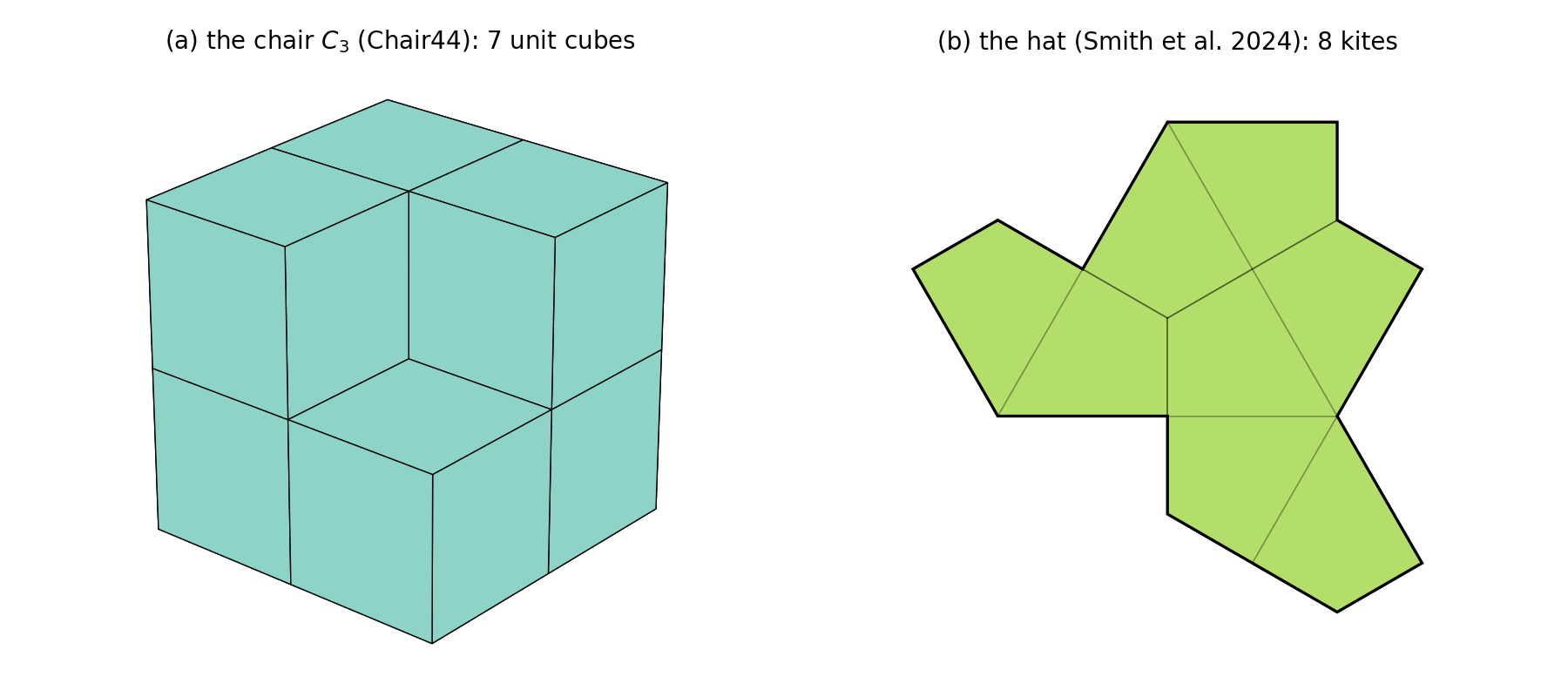}
\caption{Two aperiodic monotiles. (a) The three-dimensional chair $C_3$, the shape of the Chair44 tile: the cube $[0,2]^3$ with the unit corner cube $(1,2]^3$ removed (seven unit cubes). (b) For comparison, the planar hat of Smith, Myers, Kaplan and Goodman-Strauss \cite{hat}, a union of eight kites of the hexagonal kite grid; thin lines mark the kites.}\label{fig:chair}\end{figure}

\section{The rep-$2^N$ dissection (classical)}\label{sec:dissect}

The self-similarity of the corner-deleted hypercube is classical (Golomb \cite{golomb}; Goodman-Strauss \cite{gs99,gs2026}). We record the explicit coordinates because the frame assignment of Section~\ref{sec:frames} refers to them.

\begin{theorem}[rep-$2^N$ dissection]
For every $N \ge  2$ the scaled tile $2C_N = [0,4]^N \setminus (2,4]^N$ is the union, with disjoint interiors, of $2^N$ congruent copies of $C_N$:
\[2C_N = T_0 \cup \bigcup_{v \in U_N} T_v, \quad T_0 = C_N + (1,\dots,1), \quad T_v = 2v + D_{s(v)} C_N + c(v),\]
where $D_{s(v)}$ is the diagonal matrix with entries $s_i = -1$ if $v_i = 1$ and $+1$ otherwise, and $c(v)$ is the translation placing the reflected tile inside the block $B_v = 2v + [0,2]^N$ (explicitly $c(v)_i = 2$ if $v_i = 1$ and $0$ otherwise).
\end{theorem}

\begin{proof}
Subdivide $2C_N$ into the $2^N - 1$ blocks $B_v$, all meeting at $p_0 = (2,\dots,2)$. The central child $T_0 = [1,3]^N \setminus (2,3]^N$ lies in $2C_N$ since its missing corner $(2,3]^N \subset (2,4]^N$, and it meets each block $B_v$ in exactly one unit cube $u_v = 2v + (1-v) + [0,1]^N$. The remainder $B_v \setminus T_0$ is a $2^N$ cube missing the corner cube at $2v + (1-v)$, i.e. the corner of $B_v$ pointing towards $p_0$; reflecting $C_N$ by $D_{s(v)}$ moves its socket to that corner. Counting gives $1 + (2^N - 1) = 2^N$ children.
\end{proof}

Note that the reflection $D_{s(v)}$ has determinant $(-1)^{|v|}$; composing it with a coordinate permutation of the right parity yields a proper rotation, so the substitution can be realised homochirally. Which permutation is chosen is exactly the freedom exploited below.

\section{Frame markings and the contact language}\label{sec:frames}

Fix $N$ and identify the unit cubes of $C_N$ with their centres $2u + 1$, $u \in U_N$ (we work at twice the scale so that all coordinates are integers). A \emph{pose} is a pair $(g, t)$ with $g \in B_N$ and $t \in \mathbb{Z}^N$; the tile in pose $(g,t)$ occupies the cubes $g(2u+1) + 2t$. We only consider \emph{lattice-registered} placements: all tiles are images of $C_N$ under $B_N$ and integer translations. (This is hypothesis (2) of Theorem~\ref{thm:cond}; for the published Chair44 it is enforced by the polyhedral shape \cite{tsiokos,flicker}, and the counts "2388 contacts" in \cite{flicker} are also lattice-registered counts.)

\textbf{Marked tile.} The marked tile is $C_N$ together with its frame $g$: two placements of the marked tile are distinguished whenever their frames differ. The \emph{relative pose} of a tile $B = (g_B, t_B)$ with respect to $A = (g_A, t_A)$ is $\mathrm{rel}(A,B) = (g_A^{-1} g_B, g_A^{-1}(t_B - t_A))$. Two tiles \emph{touch} if they are interior-disjoint and share at least one panel. A \emph{pose set} $P$ is a set of relative poses of touching tiles; a lattice-registered tiling is a \emph{$P$-tiling} if the relative pose of every touching pair lies in $P$.

\textbf{Frame assignment and substitution.} A \emph{frame assignment} is a choice $\pi_0, (\pi_v)_{v \in U_N}$ of coordinate permutations; the level-1 supertile in pose $(G, T)$ consists of the children
\[(G \pi_0, T + G(1,\dots,1)), \quad (G D_{s(v)} \pi_v, T + 4 G v) \quad (v \in U_N),\]
and level-$k$ supertiles are defined recursively. We call the assignment \emph{homochiral} if all $D_{s(v)} \pi_v$ and $\pi_0$ are rotations. Throughout we take $\pi_0 = \mathrm{\mathrm{id}}$ (the central child is a translate, as in \cite{gs2026}).

\textbf{The hierarchical language.} Let $P_1$ be the set of relative poses of touching children inside one supertile, and for a pose $p$ let $X(p)$ be the set of relative poses of touching children of two supertiles in relative pose $(g, 2t)$ when $p = (g,t)$. The \emph{hierarchical language} $P_0$ is the least fixed point $P_0 = P_1 \cup X(P_0)$; it is finite (all poses are bounded) and is exactly the set of contacts occurring in hierarchical tilings (all levels). Every hierarchical tiling is a $P_0$-tiling.

\textbf{Coarsening.} For a pose set $P$ let $\Sigma(P)$ be the set of \emph{pairwise admissible supertile poses}: poses $(G,T)$ of a level-1 supertile touching the reference supertile such that all $2^N \times 2^N$ child pairs are interior-disjoint and every touching pair has relative pose in $P$. Such a pose is \emph{aligned} if $T \in 2 \mathbb{Z}^N$ and an \emph{offset} otherwise. The \emph{coarsening closure} $P^*$ of $P_0$ is obtained by iterating $P_{i+1} = P_i \cup {(G, T/2) : (G,T) \in \Sigma(P_i) \mathrm{aligned}}$ until it stabilises; it exists only if no offset pose ever becomes admissible. We say the assignment is \emph{coarsening-closed} if $P^*$ exists and $\Sigma(P^*) = 2P^*$ (equality, not just inclusion): the rule reproduces itself exactly one level up.

\textbf{Facet rules and tightness.} Number the $N\,2^N$ panels of $C_N$. A pose $p \in P$ induces the set $F(p)$ of \emph{facet triples} $(a, b, g)$: panel $a$ of the reference tile meets panel $b$ of the tile in pose $p = (g, t)$. Let $F^* = \cup_{p \in P^*} F(p)$ be the \emph{facet rules} and $P_{F^*}$ the set of touching poses all of whose facet triples lie in $F^*$. Clearly $P^* \subseteq P_{F^*}$; the rule is \emph{tight} if $P_{F^*} = P^*$. When tight, "$P^*$-tiling" is a genuinely local condition: each facet contact is checked separately against $F^*$, so $F^*$ can be realised by keys and locks on the panels (Section~\ref{sec:poly}). Note that $F^*$ constrains the relative \emph{orientation} $g$ of the two facing panels, not just their identities---this is what an asymmetric relief on each panel encodes.

\section{A finite certificate for unique hierarchy}\label{sec:cert}

Let $T$ be a tile in a $P^*$-tiling. An \emph{enclosure} of $T$ is the set of tiles covering its $N\,2^N$ panels. Given the finite set $P^*$, all combinatorially possible enclosures (pairwise $P^*$-compatible sets of neighbours covering every panel) can be enumerated. An enclosure is \emph{dead} if the patch $T \cup E$ admits no legal second shell (no set of tiles covering all panels of all tiles of the patch, pairwise compatible), and \emph{live} otherwise; deadness is decided by an exhaustive finite search. A level-1 supertile $S$ is \emph{compatible} with a patch if every child of $S$ either belongs to the patch or is interior-disjoint from and pairwise compatible with all tiles of the patch.

\begin{proposition}[finite certificate]\label{prop:cert}
Let a frame assignment be given and suppose:

\textbf{(C1)} it is coarsening-closed: $P^*$ exists and $\Sigma(P^*) = 2P^*$;

\textbf{(C2)} the facet rules are tight: $P_{F^*} = P^*$;

\textbf{(C3)} for every live enclosure $E$ of a tile $T$: (U) exactly one level-1 supertile containing $T$ is compatible with $T \cup E$, and (Pr) for every enclosure $E'$ of its central child $T_0(S)$ compatible with $T \cup E$, all children of $S$ belong to $T \cup E \cup E'$ (if $T = T_0(S)$, all children already belong to $T \cup E$).

Then every lattice-registered $F^*$-tiling $\mathcal{T}$ of $\mathbb{R}^N$ decomposes, for every $k \ge  1$, into a unique partition into complete level-$k$ supertiles, and the partition of level $k+1$ refines to that of level $k$.
\end{proposition}

\begin{proof}
By (C2) an $F^*$-tiling is the same thing as a $P^*$-tiling. Let $T \in \mathcal{T}$. Its actual enclosure $E$ in $\mathcal{T}$ is one of the enumerated ones and is live (the tiling itself is a second shell), so by (U) there is exactly one supertile $S(T)$ containing $T$ compatible with $T \cup E$; since the true supertile containing $T$---if any---must be compatible with everything present, at most one exists. By (Pr), applied to the actual enclosure of $T_0(S(T))$ in $\mathcal{T}$, all children of $S(T)$ are present in $\mathcal{T}$. Hence every tile lies in a complete level-1 supertile, unique by (U). If $T'$ is another child of $S(T)$, then $S(T)$ is a compatible supertile containing $T'$, so $S(T') = S(T)$ by uniqueness. Thus the complete supertiles partition $\mathcal{T}$.

Two touching supertiles of this partition have all their touching child pairs in $P^*$, so their relative pose lies in $\Sigma(P^*) = 2P^*$ by (C1). Scaling by $1/2$, the supertiles form a lattice-registered $P^*$-tiling by marked copies of $C_N$ (a supertile in pose $(G,T)$ becomes the tile $(G,T/2)$; note that the scaled tile's marking is again a frame, which is why closedness had to be an equality). Induction on $k$ gives the unique nested hierarchy.
\end{proof}

The proposition supplies hypothesis (3) of Theorem~\ref{thm:cond}, and (C1) also gives hypothesis (1) (hierarchical tilings exist as limits of supertiles and are $P_0 \subseteq P^*$-tilings). Every condition is a finite computation once $N$ and the frame assignment are fixed.

\textbf{On soundness of the implementation.} Deadness is established by a depth-first search over completions of the second shell that returns "dead" only after exhausting all branches; a node limit is monitored and was never reached in the runs reported below. An optional inference rule (adding the children of a uniquely decodable enclosed tile during the search) was implemented but is switched off for the certified runs, since it presupposes (Pr); with and without it the results are identical.

\section{The three-dimensional case: an independent certificate}\label{sec:N3}

\begin{figure}[t]\centering\includegraphics[width=1.0\linewidth]{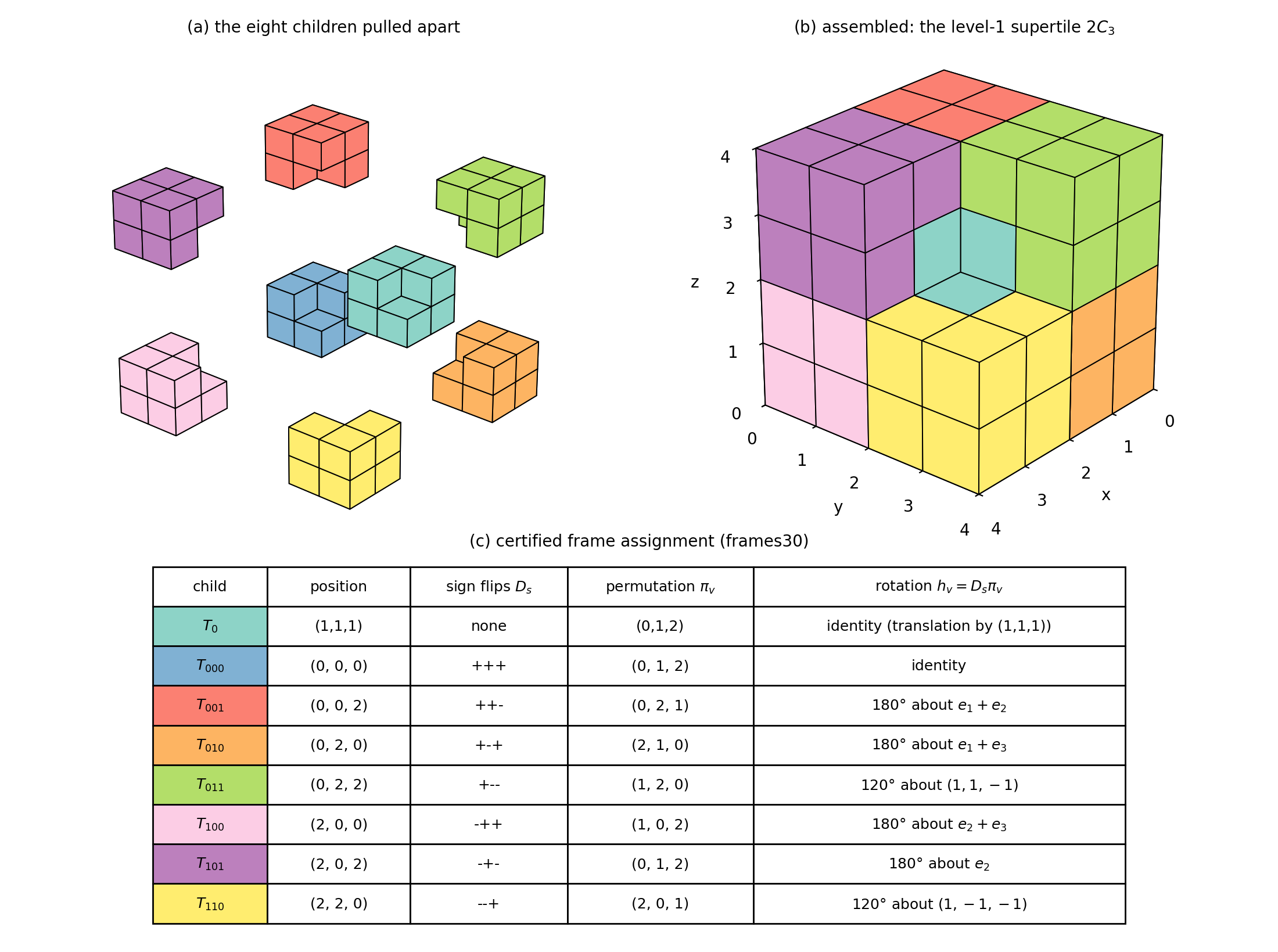}
\caption{The level-1 supertile $2C_3$ and its certified frame assignment. (a) The eight children pulled apart radially from the centre, so that every chair is visible, including the child $T_{000}$ (blue) which is hidden in the assembled supertile. (b) The same eight chairs assembled into $2C_3$, viewed from the socket side; the central translate $T_0$ (teal) is visible at the bottom of the socket. (c) The certified frame assignment. "Position" is the corner $2v$ of the $2 \times 2 \times 2$ block $B_v = 2v + [0,2]^3$ occupied by the child, i.e. the $(x,y,z)$ coordinates of the lowest corner of its bounding cube in units of unit cubes (for $T_0$ it is the translation vector $(1,1,1)$). Permutations are written as tuples $\pi = (\pi(0), \pi(1), \pi(2))$ acting on 0-indexed coordinates; the sign flips are $D_{s(v)}$. All eight child maps are proper rotations (homochiral).}\label{fig:sup}\end{figure}

\begin{theoremA}[verified matching rules for the 3D chair]
For the homochiral frame assignment of Figure~\ref{fig:sup}, conditions (C1)--(C3) hold. Consequently, with $F^*$ the induced facet rules (135 facet-contact triples on the 24 panels of $C_3$), every lattice-registered $F^*$-tiling of $\mathbb{R}^3$ is uniquely hierarchical, every such tiling is strongly aperiodic, and $|\mathrm{Aut}(\mathcal{T})| \le  24$ for each of them.
\end{theoremA}

\begin{proof}
The proof is a finite computation, organised as in Proposition~\ref{prop:cert} (2388 touching poses $\to$ 44 admissible $\to$ 33 enclosures $\to$ 15 live $\to$ one complete supertile); it is carried out by \texttt{cert6.py} in the repository, and every intermediate object is stored in \texttt{cert6\_N3triv.pkl}.

\emph{(C1)} The hierarchical language is computed as the least fixed point $P_0 = P_1 \cup X(P_0)$; it stabilises at level 3 with $|P_0| = 30$. The pairwise admissible supertile poses are enumerated by generating, for every child $a$ of the reference supertile and every $p \in P_0$, the tile $b$ in relative pose $p$ to $a$, and every supertile placement having $b$ as one of its eight children; each candidate is kept if all $64$ child pairs are interior-disjoint and all touching pairs lie in $P_0$. This gives $|\Sigma(P_0)| = 44$ poses, all aligned; halving their translations yields $P_1 = P_0 \cup 14$ new poses. Repeating with $P_1$ gives $\Sigma(P_1) = 2P_1$ exactly, so $P^* = P_1$, $|P^*| = 44$, and no offset pose is ever admissible.

\emph{(C2)} All lattice-registered poses of a second tile touching the reference tile are enumerated (2388). For each, its facet triples are computed and compared with $F^* = \cup_{p \in P^*} F(p)$, which has 135 elements; exactly the 44 poses of $P^*$ pass, so $P_{F^*} = P^*$.

\emph{(C3)} Enclosures of the reference tile are enumerated by backtracking over the panels in a fixed order, choosing for the first uncovered panel one of the tiles in $P^*$ covering it, subject to pairwise compatibility; this produces each pairwise compatible cover of the 24 panels exactly once---33 in total. For each enclosure a depth-first search attempts to complete a second shell (cover every panel of every tile of the patch with pairwise compatible tiles from $P^*$); 18 enclosures are dead (the search exhausts all branches; the node limit of $5 \cdot 10^4$ was never approached) and 15 are live. For each live enclosure the supertile placements containing the reference tile as one of its eight children are tested for compatibility with the patch; exactly one is compatible in every case (U). In one case (the reference tile is the central child) all eight children already lie in the patch. In the other 14 cases, every enclosure of the central child of that supertile compatible with the patch contains all remaining children (Pr); no compatible enclosure fails this, so no further liveness test is needed.

Hence (C1)--(C3) hold and Proposition~\ref{prop:cert} applies: every lattice-registered $F^*$-tiling is uniquely hierarchical. Existence of tilings follows from (C1): the level-$k$ supertiles are $P_0$-patches, $P_0 \subseteq P^*$, and a limit of nested supertiles containing a fixed tile is an $F^*$-tiling. Since every $g$ occurring in $P^*$ is a rotation, all tiles of an $F^*$-tiling have the same handedness. Theorem~\ref{thm:cond} now gives strong aperiodicity and $|\mathrm{Aut}(\mathcal{T})| \le  2^2 \cdot 3! = 24$.
\end{proof}

The computation produces the data of Table~\ref{tab:N3}.

\begin{table}[t]\centering\small
\begin{tabular}{p{4.5cm}p{8cm}}\toprule
\textbf{quantity} & \textbf{value} \\ \midrule
hierarchical language $|P_0|$ & 30 \\
coarsening closure & $30 \to 44 \to 44$; closed after one step, no offset poses \\
$|P^*|$ and $|\Sigma(P^*)|$ & 44 and 44 ($= 2P^*$ exactly) \\
lattice-registered touching poses & 2388 \\
facet triples $|F^*|$; tightness $|P_{F^*}|$ & 135; 44 (tight) \\
enclosures of a tile & 33 \\
dead / live & 18 / 15 \\
live enclosures with exactly one compatible supertile & 15 of 15 \\
presence (Pr) checks & 14 (the 15th is $T = T_0$); 0 failures \\
DFS node limit reached & never \\
\bottomrule\end{tabular}
\caption{Data produced by the certificate for $N = 3$ (Theorem A).}\label{tab:N3}\end{table}

\textbf{Comparison with the published constructions.} Flicker \cite{flicker} reports for Chair44: 2388 bare contacts, 44 allowed by the arrow rules of which 30 occur in infinite tilings, 33 one-shell clusters, 15 admitting a second shell, all containing the same complete "Superchair". Goodman-Strauss \cite{gs2026} describes the same mechanism in Berger's style (every marked tile lies in a marked supertile with the same corner markings; supertiles meet only fully face-to-face; $|\mathrm{Aut}| \le  24$). Our numbers coincide term by term although our marking (the full frame) and our rule (the closed contact language) were derived without reference to either paper. We regard this as an independent verification of the Chair44 mechanism at the level of the combinatorial matching rules. What the present note does \emph{not} do is re-derive the specific polyhedral shape; see Section~\ref{sec:poly}.

\textbf{Strong aperiodicity.} $P^*$ consists of proper relative rotations only, so all tiles of an $F^*$-tiling share one handedness and $L(\mathrm{Aut}(\mathcal{T})) \subseteq B_3^+$; Theorem~\ref{thm:cond} then gives $|\mathrm{Aut}(\mathcal{T})| \le  24$ and the absence of translations and screw motions.

\begin{theoremB}[rigidity of the 3D marking]
Consider all $3^7 = 2187$ homochiral frame assignments of the eight children of $2C_3$ with $\pi_0 = \mathrm{\mathrm{id}}$. Exactly three of them are coarsening-closed with $|P^*| \le  300$, namely the assignment of Figure~\ref{fig:sup} and its two conjugates under cyclic permutation of the coordinates (the assignment is invariant under conjugation by the transposition of coordinates 0 and 2). For the other 2184 the closure iteration either produces admissible offset supertile poses (6 assignments) or exceeds 300 poses within two steps without having stabilised (2178 assignments).
\end{theoremB}

\begin{proof}
This is again a finite computation (\texttt{scan3.py}, \texttt{closure\_scan3.py}). For each of the $3^7$ assignments the hierarchical language $P_0$ and the coarsening iteration are computed as in the proof of Theorem A, with two termination rules: the iteration stops with failure as soon as an offset supertile pose is admissible (6 assignments), or as soon as $|P_i| > 300$ (2178 assignments; all of these have $|P_2| \ge  398$ after two steps, whereas closure requires $|P_{i+1}| = |P_i|$). The remaining three assignments close after one step with $|P^*| = 44$; they are the assignment of Figure~\ref{fig:sup} and its images under conjugation by the two cyclic coordinate permutations, and one checks directly that conjugation by the transposition $(0 2)$ fixes the assignment of Figure~\ref{fig:sup}. The cap is part of the statement: we do not exclude that a capped iteration stabilises at some value above 300, but such a rule would admit more than six times as many contacts as the certified one, and in all capped cases the growth from $|P_1|$ to $|P_2|$ is by a factor between 3 and 7.
\end{proof}

The three closed assignments are also singled out by the simplest statistic: the number $|\Sigma(P_0)|$ of pairwise admissible supertile poses of the hierarchical language takes the values $44, 62, 99, 103, 113, 114, 119, 158, 186$ over the 2187 assignments and the minimum $44$ is attained exactly by them (\texttt{scan3.py}, \texttt{closure\_scan3.py}). The uniqueness is a structural fact: self-similarity of the rule leaves no freedom in the marking, once the central child is a translate.

\begin{figure}[t]\centering\includegraphics[width=0.72\linewidth]{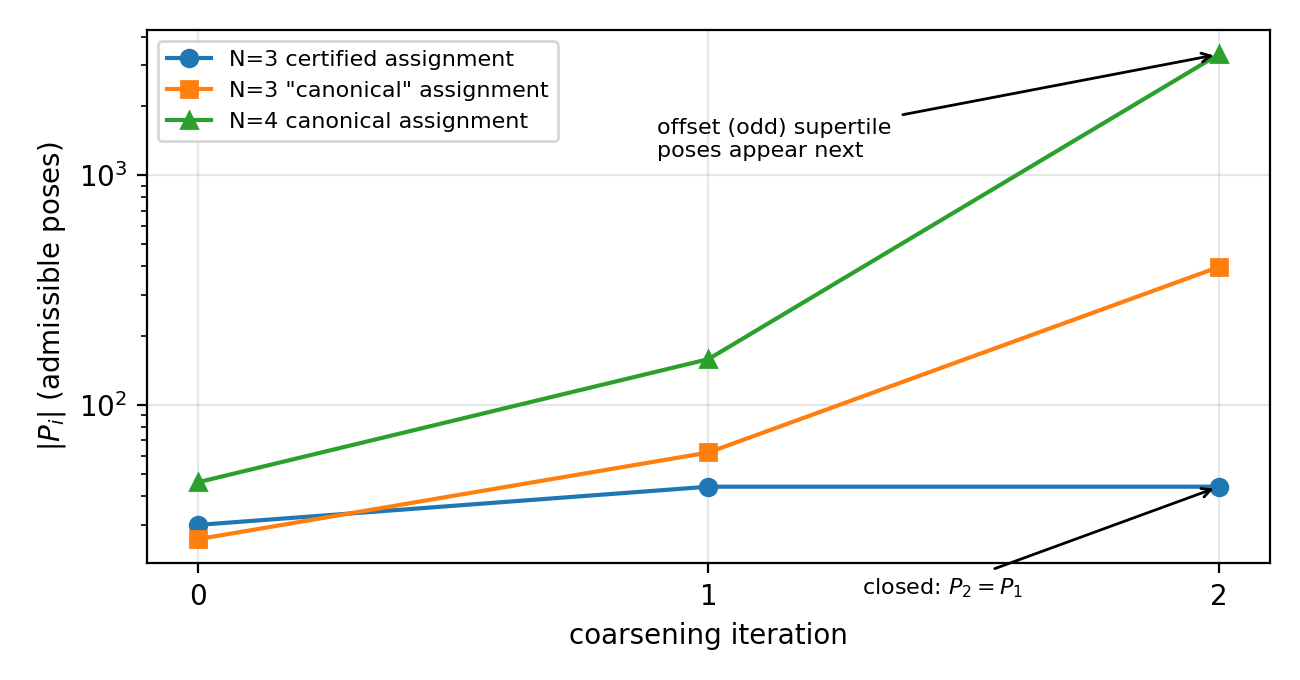}
\caption{Coarsening iteration $|P_i|$ for three assignments. The certified 3D assignment closes ($30 \to 44 \to 44$). The "canonical" assignments (sign flips composed with a fixed transposition where a rotation is required) do not: the closure grows and then admits offset supertile poses, in $N = 3$ and $N = 4$ alike.}\label{fig:closure}\end{figure}

\section{Four dimensions: what the pipeline shows}\label{sec:N4}

All code is dimension-independent; running it on $C_4$ ($16$ children, $64$ panels, $384$ frames) is a matter of seconds per frame assignment. Unlike $N = 3$, however, the assignment space is $12^15 \approx 1.5 \cdot 10^16$ homochiral choices (12 proper rotations per outer child), so exhaustive enumeration is impossible with our resources and we tested the structured families of Table~\ref{tab:N4}.

\begin{table}[t]\centering\small
\begin{tabular}{p{5.2cm}rp{3.2cm}p{4cm}}\toprule
\textbf{family} & \textbf{size} & \textbf{best $|\Sigma(P_0)|$ (min over family)} & \textbf{closure} \\ \midrule
canonical ($\pi_v = \mathrm{\mathrm{id}}$ or a fixed transposition, by parity) & 1 & 158 ($|P_0| = 46$) & $46 \to 158 \to 3354 \to$ offsets \\
$D_4$-conjugation-invariant & 32 & 158 & fails \\
$\langle (0 3), (1 2) \rangle$-invariant & 768 & 158 & fails \\
$Z_4$-invariant & 6912 & 158 & fails \\
suspension of the 3D solution ($\pi_v = \pi^{(3)}_{v'}$ or $(3 a) \pi^{(3)}_{v'}$) & 6 & 840 & fails \\
coordinate descent on $|\Sigma(P_0)|$, random starts & --- & plateau 1566 & fails \\
\bottomrule\end{tabular}
\caption{Families of frame assignments of $C_4$ tested for coarsening closure (C1). For the certified 3D assignment the corresponding value of $|\Sigma(P_0)|$ is $44 = 2|P_0| - 16$; all 4D families stay far above $2|P_0|$.}\label{tab:N4}\end{table}

Two facts about the failures are worth recording. First, in every case examined the \emph{first} coarsening step produces only aligned supertile poses; offsets appear only after the closure has grown (Figure~\ref{fig:closure}). In other words, the pairwise child rules do force supertiles to meet face-to-face, and what fails is that the enlarged rule no longer coincides with the original one---a failure of \emph{self-similarity} of the marking rather than of local rigidity. Second, the objective $|\Sigma(P_0)|$ has very large plateaus (value 1566 for generic assignments, 158 for the canonical-type ones), which defeats local search; in $N = 3$ the same plateau (value 158) contains 693 of the 2187 assignments, while the solution is an isolated minimum.

\begin{problem}[self-similar marking of $C_4$]\label{prob:N4}
Decide whether there is a homochiral frame assignment $(\pi_v)_{v \in U_4}$ of the 16 children of $2C_4$, with $\pi_0 = \mathrm{\mathrm{id}}$, that is coarsening-closed ((C1)); and, if so, whether (C2) and (C3) hold for it. A positive answer, by Proposition~\ref{prop:cert} and Theorem~\ref{thm:cond}, gives a marked strongly aperiodic monotile in $\mathbb{R}^4$ with $|\mathrm{Aut}| \le  192$.
\end{problem}

Goodman-Strauss \cite{gs2026} remarks that the shape and substitution of Chair44 generalise to all dimensions and that his 1999 construction \cite{gs99} gives an aperiodic \emph{pair} in $\mathbb{R}^N$; whether the single-tile marking generalises is exactly Problem~\ref{prob:N4}. Our Theorem B shows that in $N = 3$ the marking is forced, so the question is not one of choice but of existence. We note two natural reductions: (i) a solution invariant under conjugation by a subgroup $\Gamma \le  S_4$ can be enumerated over orbits (this is what \texttt{symsearch4.py} does; the $3$-dimensional solution has $\Gamma = Z_2$); (ii) condition (C1) can be tested on a candidate in about two seconds, so any conceptual proposal for the marking (for instance a corner-colouring rule in the style of \cite{gs2026}) can be checked immediately with the supplied code.

\begin{conjecture}[higher-dimensional polyhedral monotiles]
For each $N \ge  4$ there exists an unmarked polyhedral prototile $Q_N \subset \mathbb{R}^N$ based on $C_N$ satisfying hypotheses (1)--(3) of Theorem~\ref{thm:cond}.
\end{conjecture}

We no longer offer an "$N$-colour corner-cyclic grammar" as a proposed rule: the naive corner-colouring analogue we tested (colours at corner facet-germs, equality on coincident germs) does not even reproduce $P^*$ for $N = 3$ (it accepts 2016 of the 2388 bare contacts that carry no corner-to-corner coincidence at all). The frame formalism replaces it by a rule that is derived rather than guessed.

\section{Polyhedral realisation via keys and locks}\label{sec:poly}

Following the standard encoding technique (Goodman-Strauss \cite{gs98,gs2026}), a facet-pair relation such as $F^*$ can be converted into rigid boundary geometry.

\begin{construction}[keys and locks]
Let $F_j \subset \partial C_N$ ($j = 1, \dots, N\,2^N$) be the panels. (1) On each panel place an asymmetric $(N-1)$-dimensional relief of height $\varepsilon < 1/8$ (a protrusion on "key" panels, a matching depression on "lock" panels) whose shape depends on the panel index. (2) Choose the reliefs so that panel $a$ of one tile fits flush against panel $b$ of another, in relative orientation $g$, if and only if $(a, b, g) \in F^*$, and so that no partial (off-lattice) overlaps of reliefs are possible. (3) Set $Q_N := (C_N \setminus \bigcup_{\mathrm{locks}} P_j^-) \cup \bigcup_{\mathrm{keys}} P_j^+$.
\end{construction}

\begin{figure}[t]\centering\includegraphics[width=1.0\linewidth]{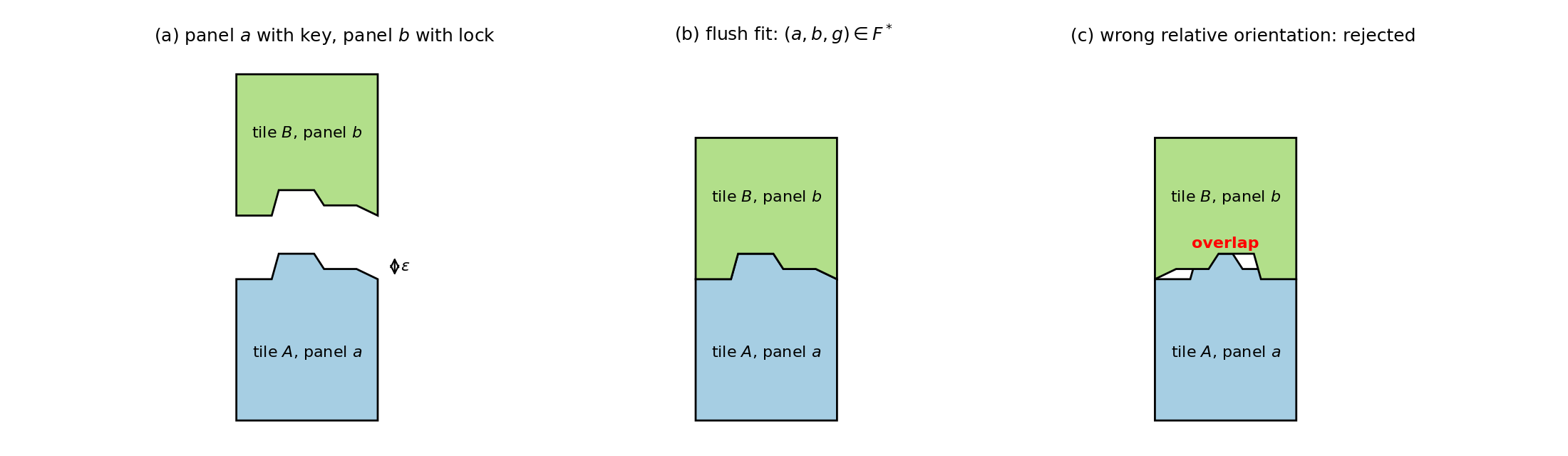}
\caption{Keys and locks in cross-section. (a) An asymmetric relief of height $\varepsilon$ on panel $a$ (key) and its complement on panel $b$ (lock). (b) The two panels fit flush only in the relative orientations $g$ for which $(a,b,g) \in F^*$. (c) In any other relative orientation the reliefs overlap, so the placement is impossible.}\label{fig:keys}\end{figure}

Figure~\ref{fig:keys} illustrates the construction. Step (2) is where the actual work lies; for $N = 3$ the published Chair44 shape \cite{tsiokos} is such a realisation of an equivalent rule set (Flicker \cite{flicker} shows the arrow rules and Chair44 are equivalent). We do not claim a new polyhedron here: Theorem A is a statement about a \emph{marked} tile with facet matching rules, and lattice registration is an assumption in our model that a polyhedral realisation must enforce by shape.

\section{Conditional strong aperiodicity}\label{sec:cond}

\begin{theorem}[strong aperiodicity under hierarchical enforcement and lattice registration]\label{thm:cond}
Let $Q_N \subset \mathbb{R}^N$ be a prototile (possibly marked) and assume:
(1) $Q_N$ admits tilings of $\mathbb{R}^N$ containing arbitrarily large supertiles;
(2) every admitted tiling is lattice-registered, so that $L(\mathrm{Aut}(\mathcal{T})) \subseteq B_N$ for the linear part $L$;
(3) every admitted tiling decomposes into a unique nested supertile hierarchy ${\mathcal{S}_k}_{k \ge  0}$ with inflation factor 2.
Then for every admitted tiling $\mathcal{T}$: $\mathrm{Aut}(\mathcal{T})$ contains no non-zero translation and no element of infinite order; the homomorphism $L : \mathrm{Aut}(\mathcal{T}) \to O(N)$ is injective; hence $|\mathrm{Aut}(\mathcal{T})| \le  |B_N| = 2^N N!$, and $\le  2^{N-1} N!$ if all tiles have the same handedness.
\end{theorem}

\begin{proof}
Let $\gamma(x) = R x + v \in \mathrm{Aut}(\mathcal{T})$. By (3), $\gamma$ preserves each level-$k$ partition. By (2), $R \in B_N$ has finite order $m$, and $\gamma^m(x) = x + w$ with $w = \sum_{j=0}^{m-1} R^j v$. Suppose $w \ne  0$. Supertiles of level $k$ contain balls of radius $r(S_k) \ge  2^k r_0 - \varepsilon \to \infty$, so for large $k$ some $S \in \mathcal{S}_k$ contains a point $x$ with $\mathrm{dist}(x, \partial S) > ||w||$; then $x + w \in \mathrm{int}(S)$, so $\gamma^m (S) \cap \mathrm{int}(S) \ne  \emptyset$, and since $\gamma^m$ permutes the interior-disjoint level-$k$ supertiles, $\gamma^m (S) = S$, i.e. $S + w = S$. Taking $M = \sup_{y \in S} \langle y, w \rangle$, finite because $S$ is bounded and non-empty, $S + w = S$ gives $M = M + ||w||^2$, so $w = 0$: a contradiction. Hence $\gamma^m = \mathrm{\mathrm{id}}$; in particular a pure translation ($m = 1$) is trivial, and $\ker L = {\mathrm{\mathrm{id}}}$. Injectivity of $L$ embeds $\mathrm{Aut}(\mathcal{T})$ into $B_N$, giving the bounds; if all tiles have the same handedness, $L(\mathrm{Aut}(\mathcal{T})) \subseteq B_N^+$.
\end{proof}

For $N = 3$ and the marked tile of Theorem A, hypotheses (1) and (3) are supplied by Proposition~\ref{prop:cert}, and (2) is part of the model.

\section{Dimensional summary}

Table~\ref{tab:dim} summarises the situation dimension by dimension.

\begin{table}[t]\centering\small
\begin{tabular}{cp{2.6cm}ccc p{4.3cm}}\toprule
\textbf{$N$} & \textbf{prototile} & \textbf{volume} & \textbf{children} & \textbf{panels} & \textbf{status / bound} \\ \midrule
2 & L-tromino & 3 & 4 & 8 edges & 2-tile enforcement known \cite{gs99b} \\
3 & Chair44; marked $C_3$ of Theorem A & 7 & 8 & 24 & verified; $|\mathrm{Aut}| \le  24$ \\
4 & tesseract chair & 15 & 16 & 64 & Problem~\ref{prob:N4}; conditional bound 192 \\
5 & penteract chair & 31 & 32 & 160 & open; conditional bound 1920 \\
$N$ & hypercube chair & $2^N - 1$ & $2^N$ & $N\,2^N$ & conditional bound $2^{N-1} N!$ \\
\bottomrule\end{tabular}
\caption{The chair in dimensions 2--5 and the status of matching rules for it.}\label{tab:dim}\end{table}

\section{Future work}

\begin{enumerate}
\item \textbf{Settle Problem~\ref{prob:N4}.} The decisive open question is whether a coarsening-closed frame assignment of $C_4$ exists. Two routes look feasible: (a) a conceptual derivation of the marking in the spirit of Goodman-Strauss's corner/socket description \cite{gs2026}, which would give a candidate for every $N$ at once and can be tested with \texttt{cert6.py} in seconds; (b) a complete search of the $12^15$ homochiral assignments using stronger pruning, e.g. by requiring closure of the restricted rule on each three-dimensional face of the supertile, or by a SAT/SMT encoding of condition (C1). Theorem B suggests that a solution, if it exists, is isolated and highly structured, so symmetry-based enumeration over larger subgroups $\Gamma \le  S_4$ is worth completing.

\item \textbf{Relax the assumptions of the model.} Our certificate assumes lattice registration and translated central child ($\pi_0 = \mathrm{\mathrm{id}}$). Dropping the second is a factor-24 larger search; dropping the first requires enumerating off-lattice contacts, as in the polyhedral realisations of \cite{tsiokos,flicker}.

\item \textbf{From marked tile to polyhedron.} Carry out the key-and-lock construction of Section 10 explicitly for the rules $F^*$ of Theorem A, and prove that the resulting polyhedron forces lattice registration; compare the result with the Chair44 shape.

\item \textbf{Smaller rules.} The 135 facet triples of $F^*$ are what the full-frame marking produces; the minimal number of distinct facet labels (and whether the 14 admissible-but-never-occurring poses of Appendix A can be excluded by a local rule) is open.

\item \textbf{Other rep-tiles.} The formalism uses nothing specific to the chair beyond a substitution with a lattice-registered patch; applying it to other $N$-dimensional rep-tiles (e.g. the notched cubes of \cite{gs99}) may produce further marked monotiles or show that self-similarity of the marking fails there too.
\end{enumerate}

\section{Conclusion}

The frame-marking formalism turns the search for chair matching rules into finite computations in every dimension. In three dimensions it recovers the Chair44 mechanism from first principles, certifies it in two seconds, and shows the marking is forced by self-similarity. In four dimensions the same computations show that the obvious generalisations are \emph{not} self-similar and isolate the precise finite problem that remains. We hope the supplied code lowers the cost of settling Problem~\ref{prob:N4}.

\bibliographystyle{plain}
\bibliography{refs}

\section*{Appendix A: the 44 admissible relative poses for $N = 3$}

Relative poses $(g, t)$ of a touching tile with respect to the reference tile (cubes at $2u + 1$, $u \in U_3$); $g$ is written by the images of $x, y, z$, e.g. $(+y -z -x)$ maps $(x,y,z) \mapsto (y, -z, -x)$; "H" marks the 30 poses that occur in hierarchical tilings, "·" the 14 admitted by the rules but never occurring.

\begin{multicols}{3}\scriptsize\begin{verbatim}
· t=(-2, 2, 2) g=(+x-y-z)
H t=(-2, 2, 2) g=(+y-z-x)
· t=(-2, 2, 2) g=(+z-x-y)
H t=(-1, -1, -1) g=(+x+y+z)
H t=(-1, -1, 3) g=(+x+z-y)
H t=(-1, 3, -1) g=(+z-y+x)
H t=(-1, 3, 3) g=(+y-z-x)
H t=(0, 0, 0) g=(-z+y+x)
H t=(0, 0, 0) g=(+z-y+x)
H t=(0, 0, 0) g=(+z+y-x)
H t=(0, 0, 4) g=(+x+z-y)
H t=(0, 0, 4) g=(+y+x-z)
· t=(0, 0, 4) g=(+z+y-x)
H t=(0, 4, 0) g=(+x-z+y)
H t=(0, 4, 0) g=(+y-x+z)
H t=(0, 4, 0) g=(+z-y+x)
H t=(1, 1, 1) g=(+x+y+z)
H t=(1, 3, 1) g=(+x-z+y)
H t=(1, 3, 1) g=(+y-x+z)
H t=(1, 3, 1) g=(+z-y+x)
H t=(2, -2, 2) g=(-x+y-z)
· t=(2, -2, 2) g=(-y+z-x)
· t=(2, -2, 2) g=(-z+x-y)
· t=(2, 2, -2) g=(-x-y+z)
· t=(2, 2, -2) g=(-y-z+x)
H t=(2, 2, -2) g=(-z-x+y)
· t=(2, 2, 2) g=(-x-y+z)
H t=(2, 2, 2) g=(-x+y-z)
· t=(2, 2, 2) g=(+x-y-z)
· t=(2, 2, 2) g=(-y-z+x)
H t=(2, 2, 2) g=(-y+z-x)
· t=(2, 2, 2) g=(+y-z-x)
· t=(2, 2, 2) g=(-z-x+y)
H t=(2, 2, 2) g=(-z+x-y)
· t=(2, 2, 2) g=(+z-x-y)
H t=(3, -1, -1) g=(-y+x+z)
H t=(3, -1, 3) g=(-x+y-z)
H t=(3, 1, 3) g=(-x+y-z)
H t=(3, 1, 3) g=(-y+z-x)
H t=(3, 1, 3) g=(-z+x-y)
H t=(3, 3, -1) g=(-z-x+y)
H t=(4, 0, 0) g=(-x+z+y)
H t=(4, 0, 0) g=(-y+x+z)
· t=(4, 0, 0) g=(-z+y+x)
\end{verbatim}\end{multicols}

\section*{Appendix B: human inquiry record}

This research was conducted through an iterative human--AI collaboration on Arena.ai. Chronological summary of the human prompts by Dmitry Kamenetsky and of what the AI changed in response:
\begin{itemize}
\item \textbf{Prompt 1 (problem conception).} "recently chair44 was discovered. A 3D shape that can tile space non-periodically. Investigate whether the same techniques can be extended to any N-D space." → The AI analysed the Chair44 mechanism, reviewed the corner-deleted hypercube $C_N$, and initiated the multidimensional extension study.
\item \textbf{Prompt 2 (preprint and novelty).} "excellent! Is this a novel result? If so, can you prepare an arXiv paper about it with references and proofs." → Coordinate proofs, literature comparison, and a first formatted preprint.
\item \textbf{Prompt 3 (authorship and diagrams).} Add the author and the AI to the author list; state that the work was done by the AI on the author's suggestion; add diagrams of the $N$-dimensional tiles and tilings; produce a PDF. → Authorship metadata and contribution statement, vector diagrams, Typst compilation toolchain.
\item \textbf{Prompt 4 (refinement and citations).} Reference numbers for Goodman-Strauss and Flicker; quotation marks in "einstein"; repositioned figure text; expanded boxes; novelty text; clipping fix; cleaned references. → Figure text placement, citation anchors, diagram clipping, metadata.
\item \textbf{Prompt 5 (tiling visualisations, appendix, verification).} Figures of actual tilings; this appendix; wording of the contribution statement; verification of all proofs and formatting. → Space-filling patches for 2D, 3D and 4D cross-sections, Appendix A of that version, mathematical audit.
\item \textbf{Prompts 6--7 (external review).} "I gave this paper to another AI to review and it said that it cannot be published. Here is its feedback. Can you address the feedback and fix the paper? Can we still salvage some of it that is novel for a publication?" → Corrected $T_0$; fixed the inradius/supremum chain so that it proves $\gamma^m = \mathrm{\mathrm{id}}$; derived finiteness and the symmetry bound via the injective projection $L : \mathrm{Aut}(\mathcal{T}) \to O(N)$; cited Goodman-Strauss (1999); renumbered references; reframed $N \ge  4$ as a programme; aligned all claims with the expository scope.
\item \textbf{Prompt 8 (research contribution).} "Can you add the corrections and make it a genuine research contribution." → The frame-marking formalism, the finite certificate (Proposition~\ref{prop:cert}), the $N = 3$ certificate (Theorem A) and rigidity scan (Theorem B), the $N = 4$ computations and Problem~\ref{prob:N4}, the accompanying code.
\item \textbf{Prompt 9 (presentation).} Shorter title; removal of the internal "what is new" subsection in favour of a contributions section; typographic fixes (theorem environments, section numbering, stray symbols); proofs of Theorems A and B written out; unbreakable tables; GitHub repository; appendices after the references; figure of the supertile re-oriented to show the socket. → All applied; repository \texttt{dimkadimon/Monotile-RN} created and linked.
\item \textbf{Prompt 10 (completion of the $Z_4$ search).} "Yes update the table once the code finishes running." → The $Z_4$-symmetric family (6912 assignments) was completed and the table updated; duplicated cross-reference words removed; version number, figure caption and future-work section added.
\item \textbf{Prompt 11 (accessibility).} Keys-and-locks figure restored; contributions section rewritten for a general audience; two-line title; matching colours in Figure 1; table captions and references; this detailed record.

\item \textbf{Prompt 12 (novelty and LaTeX).} Each contribution in Section~\ref{sec:contrib} tagged as new or previously known, with the relevant citations; Figure 1 caption corrected ($T_{000}$ is blue); a LaTeX source of the paper generated from the Typst source for arXiv submission (\texttt{paper.tex}, \texttt{refs.bib}, in the repository).

\item \textbf{Prompt 13 (figure clean-up).} The tilings figure was judged confusing and removed; the blocks in Figure 1 were made semi-transparent.

\item \textbf{Prompt 14 (chair figure).} Figure 1 reverted to opaque tiles; a new figure of a single chair $C_3$ (the Chair44 shape) added in Section~\ref{sec:polytope}; "Thm. Theorem A" corrected in the dimensional table; the code-availability sentence moved to the end of the abstract.

\item \textbf{Prompt 15 (hat and exploded view).} The chair figure reduced to a single opaque chair, with the hat added for comparison; an exploded view of the eight children added.

\item \textbf{Prompt 16 (merged figure).} The exploded view, the assembled supertile and the frame table merged into one figure (Figure~\ref{fig:sup}); $T_0$ displaced in the exploded view so that $T_{000}$ is visible; a rendering artefact at the contact of $T_0$ and $T_{011}$ removed by drawing all faces in a single depth-sorted collection.

\item \textbf{Prompt 17 (this version).} The single chair in Figure~\ref{fig:chair} re-oriented so that the socket faces the reader, matching $T_0$ in Figure~\ref{fig:sup}; arXiv subject classes suggested.
\end{itemize}

\end{document}